\documentclass{amsart}
\usepackage{amsmath,mathrsfs}
\usepackage[utf8]{inputenc}
\usepackage[all,cmtip]{xy}
\usepackage{tikz}
\usetikzlibrary{matrix,arrows}

\usepackage{amssymb}
\usepackage{amscd,amsbsy}
\usepackage{hyperref}
\hypersetup{
    bookmarks=true,         % show bookmarks bar?
    unicode=false,          % non-Latin characters in Acrobat?s bookmarks
    pdftoolbar=true,        % show Acrobat?s toolbar?
    pdfmenubar=true,        % show Acrobat?s menu?
    pdffitwindow=false,     % window fit to page when opened
    pdfstartview={FitH},    % fits the width of the page to the window
    pdftitle={My title},    % title
    pdfauthor={Author},     % author
    pdfsubject={Subject},   % subject of the document
    pdfcreator={Creator},   % creator of the document
    pdfproducer={Producer}, % producer of the document
    pdfkeywords={keyword1} {key2} {key3}, % list of keywords
    pdfnewwindow=true,      % links in new window
    colorlinks=true,       % false: boxed links; true: colored links
    linkcolor=black,          % color of internal links
    citecolor=black,        % color of links to bibliography
    filecolor=black,      % color of file links
    urlcolor=black           % color of external links
}
\usepackage{longtable}
\usepackage[activate={true,nocompatibility}]{microtype}
\newtheorem{theorem}{Theorem}[section]
\newtheorem*{problem}{Problem}
\newtheorem{lemma}[theorem]{Lemma}
\newtheorem{proposition}[theorem]{Proposition}

\theoremstyle{definition}
\newtheorem{definition}[theorem]{Definition}

\newtheorem{definitions and remarks}[theorem]{Definitions and Remarks}

\theoremstyle{remark}
\newtheorem{remark}[theorem]{Remark}

\numberwithin{equation}{section}

\begin{document}
\title[Generelized Flow-Box]{Generalized Flow-Box property for singular foliations}

\author[A.~Belotto]{Andr\'e Belotto da Silva}
\author[D.~Panazzolo]{Daniel Panazzolo}

\thanks{Universit\'e Paul Sabatier, Institut de Math\'ematiques de Toulouse, UMR CNRS 5129, 118 route de Narbonne, F-31062 Toulouse Cedex 9, France ({\tt andre.belotto\_da\_silva@math.univ-toulouse.fr}).}
\thanks{Université de Haute Alsace, {IRIMAS, Département de Mathématiques}, 4 Rue des Frères Lumière - 68093 Mulhouse, France and Université de Strasbourg, France ({\tt daniel.panazzolo@uha.fr})}

\dedicatory{Dedicado \`{a} Felipe Cano, em comemora\c{c}\~{a}o dos seus 60 anos.} 


\begin{abstract}
We introduce a notion of generalized Flow-Box property valid for general singular distributions and sub-varieties (based on a dynamical interpretation). Just as in the usual Flow-Box Theorem, we characterize geometrical and algebraic conditions of (quasi) transversality in order for an analytic sub-variety $X$ (not necessarily regular) to be a section of a line foliation. We also discuss the case of more general foliations.

This study is originally motivated by a question of Jean-Fran\c{c}ois Mattei {in \cite[Theorem 3.2.1]{Mattei}} about the existence of local slices for a (non-compact) Lie group action.

\end{abstract}

\maketitle
\setcounter{tocdepth}{1}
%\tableofcontents

\section{Introduction}\label{sec:intro}
\noindent 
The Flow-Box Theorem (also called Straightening Theorem) stands as an important classical tool for the study of vector-fields. The Theorem states that the dynamic near a non-singular point is as simple as possible, that is, it is conjugated to a translation (see e.g. \cite[Theorem 1.14]{IY}). The Frobenius Theorem can be seen as a generalization for general foliations. In this work, we introduce a notion of generalized Flow-Box based on a dynamical interpretation and which is valid for singular foliations. We show, furthermore, that under ``quasi-transversality" conditions the generalized Flow-Box property holds. We were originally motivated by the study of local slices for (non-compact) Lie group actions and a question of {J.-F}. Mattei which we briefly discuss on subsection \ref{ssec:Or} below.

Let us denote by $M$ a real or complex analytic manifold and by $X$ an analytic (not necessarily regular) sub-variety of $M$. For simplicity, we specialize our presentation for vector-fields before considering general involutive singular distributions in subsection \ref{ssec:InM}.
\subsection{Vector-fields} \label{ssec:InV}
Let $\partial$ be an analytic vector-field over $M$ and denote by $Sing(\partial)$ the set of singular points of $\partial$. We say that the triple $(M,\partial,X)$ is \emph{transverse} (or, more commonly, that $X$ is a \emph{transverse section} of $\partial$) if the vector-field $\partial$ is transverse to $X$, that is, at every point $p\in X$ there exists an analytic germ {of function} $f$ such that $X \subset (f=0)$ and $\partial(f){(p) \neq 0}$. We consider the following version of the Flow-Box Theorem:

\begin{theorem}[Flow-Box Theorem]
If a triple $(M,\partial,X)$ is transverse then, at each point $p\in X$, there exists a neighbourhood $U$ of $p$ and a coordinate system $(x,\pmb{y})=(x,y_2,\ldots, y_m)$ over $U$ centered at $p$ such that $\partial= \partial_x$ and $X\cap U \subset \{x=0\}$.
\end{theorem}

\begin{remark}
\label{rk:FLBOX}
\begin{enumerate}
\item Sometimes the Flow-Box Theorem (see e.g. \cite[Theorem 1.14]{IY}) is stated independently of a transverse section $X$, but the above version is an equivalent statement (which follows from the Implicit Function Theorem).
\item Let $p\in X$ and denote by $\varphi_p(t)$ the orbit of $X$ such that $\varphi_p(0)=p$. If $(M,\partial,X)$ is transverse, then for every compact set $K \subset X$, there {exists $\delta>0$} such that $\varphi_p(t) \notin X$ for all $0<|t|<\delta$ and for all $p\in K$.
\item If $M$ is a complex manifold and $X$ is a codimension one smooth sub-variety, then {the reciprocal of statement in $(2)$ also holds}.
%If $M$ is a complex manifold and $X$ is a codimension one smooth sub-variety, then {the statement in $(2)$ implies that $(M,\partial,X)$ is transverse}. 
\end{enumerate}
\end{remark}

Based on the interpretation given in remark \ref{rk:FLBOX}(2), we define:
\begin{definition}[Generalized Flow-Box property]\label{def:gFB}
We say that the triple $(M,\partial,X)$ satisfies the \textit{generalized Flow-Box property} if for every compact set $K \subset X$, there exists a real number $\delta>0$ such that for every point $p \in K \setminus Sing(\partial)$ we have that $\varphi_p(t) \notin X$ for all $0<|t|<\delta$ (where $\varphi_p(t)$ is the orbit of $\partial$ such that $\varphi_p(0) = p$).
\end{definition}
\noindent
Note that the generalized Flow-Box property is defined for singular sub-varieties $X$ and vector-fields $\partial$. In order to illustrate the property, we present two examples:
\begin{enumerate}
\item Consider the triple $\left(\mathbb{R}^2,\partial = x\partial_y - y\partial_x, (y=0)\right)$ (see figure \ref{fig:Ex}). In this case $(M,\partial,X)$ satisfies the generalized Flow-Box property with, for example, $\delta= \pi/4$;
\item Consider the triple $\left(\mathbb{R}^2,\partial = \partial_x + x\partial_y, (y=0)   \right)$ (see figure \ref{fig:Ex}). In this case $(M,\partial, X)$ does not satisfy the generalized Flow-Box property;

\end{enumerate}
\begin{figure}
  \centering
\includegraphics[scale=0.7]{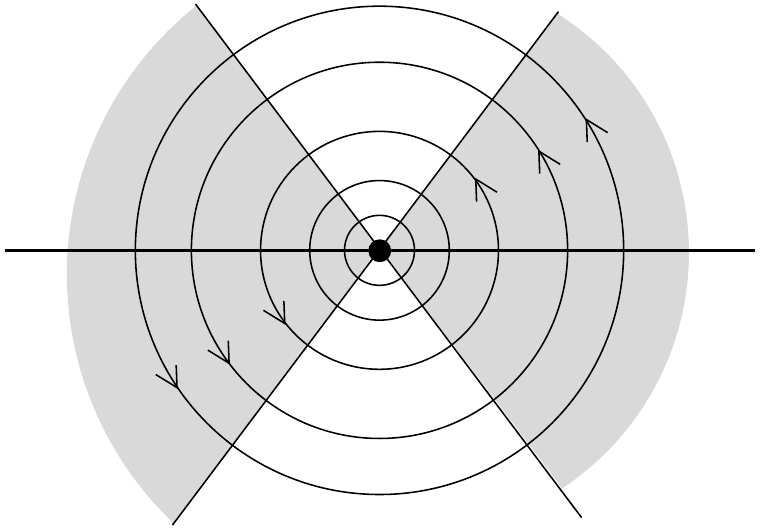}\includegraphics[scale=0.9]{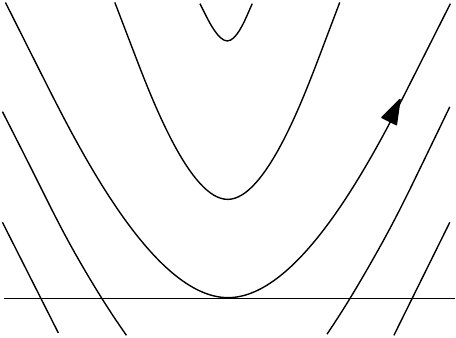}
\caption{\small{In the left, the phase space of Example 1. In the right the phase space of Example 2.}}
\label{fig:Ex}
\end{figure}
Now, just as in the usual Flow-Box Theorem, we characterize this condition in terms of a transversality condition between $\partial$ and $X$. We start by a geometrical condition:
\begin{definition}
A triple $(M,\partial,X)$ is \emph{geometrically quasi-transverse} if the triple $(M,\partial, X\setminus Sing(\partial))$ is transverse, that is, if at every point $p \in X\setminus Sing(\partial)$, there exists an analytic germ $f$ such that $X \subset (f=0)$ and $\partial(f)$ is a unit (at $p$).
\end{definition}

Under this condition, we prove the following result:
\begin{theorem}
Let $(M,\partial,X)$ be a geometrically quasi-transverse triple and assume that $X$ is non-singular. 
\begin{itemize}
\item[(a)] If $X$ has dimension or codimenstion one, then $(M,\partial,X)$ satisfies the generalized Flow-Box property. 
\item[(b)] If $dim \, M\geq 4$, there exists a geometrically quasi-transverse triple $(M,\partial,X)$ (where $X$ is non-singular) that does not satisfy the generalized Flow-Box property.
\end{itemize}
\label{th:R2simpl}
\end{theorem}
\noindent
Condition $(a)$ of the above result is a consequence of Theorem \ref{th:R1} below and the part $(b)$ is proved in subsection \ref{ssec:Example}. Next, we introduce a stronger notion of quasi-transversality in order to get a sufficient condition:

\begin{definition}
A triple $(M,\partial,X)$ is \emph{analytically quasi-transverse} if it is geometrically quasi-transverse and $\partial^2(\mathcal{I}_X) \subset \partial(\mathcal{I}_X)+\mathcal{I}_X$, where $\mathcal{I}_X$ is the reduced ideal sheaf whose support is $X$.
\end{definition}

The above condition avoids tangent contacts between $X$ and $\partial$ which are ``hidden" in the singular locus. Indeed:

\begin{theorem}
If a triple $(M,\partial,X)$ is analytically quasi-transverse, then it satisfies the generalized Flow-Box property.
\label{th:R1simpl}
\end{theorem}
\noindent
The above result is a consequence of Theorem \ref{th:R1} below.

\subsection{Involutive Singular Distributions} 
\label{ssec:InM}
Following \cite[p. 281]{BaumBott}, an analytic singular foliation is a coherent sub-sheaf $\theta$ of $Der_M$ (where $Der_M$ is the sheaf of all derivations over $M$) which is closed under the Lie-Bracket operation. We say that a triple $(M,\theta,X)$ satisfies the \textit{Flow-Box} property at a point $p$ in $X$ if there exists a neighbourhood $U$ of $p$ and a coordinate system $(\pmb{x},\pmb{y})=(x_1,\ldots, x_t,y_{t+1},\ldots, y_m)$ over $U$ centered at $p$ such that $\theta$ is generated by $\{ \partial_{x_1}, \ldots, \partial_{x_t}\}$ and $X\cap U$ is a subset of $V(x_1,\ldots, x_t) {:=\{x_1 =\ldots = x_t =0\}}$. We say that the triple $(M,\theta,X)$ is \emph{transverse} (or, more commonly, $X$ is a transverse section of $\theta$) if $X$ is a regular sub-variety and at every point $p$ in $X$, the leaf $\mathcal{L}$ of $\theta$ passing through $p$ has maximal dimension and its tangent space has trivial intersection with the tangent space of $X$. The classical {Frobenius} theorem states that any triple $(M,\theta,X)$ which is transverse, satisfies the Flow-Box property. Following the same considerations which lead to the definition \ref{def:gFB}, we define:
\begin{definition}\label{def:gFBM}
We say that a triple $(M,\theta,X)$ satisfies the \emph{generalized Flow-Box property} at a point $p$ in $X$, if there exists a triple $(U,\delta,g)$, where $U$ is an open neighbourhood of $p$, $\delta>0$ is a real number and $g$ is a sub-Riemannian metric {(whose definition is recalled in subsection \ref{ssec:SRG})} over $M$ generated by a set $\{\partial_1, \ldots, \partial_d\}$ of local generators of $\theta$ such that: for every point $q \in U \cap X$, every $g$-ball with radius $\eta<\delta$:
\[
B_{\eta}^g(q) = \{a \in U ; d_g(a,q) <\eta\}
\]
intercepts $X$ only at $q$ and is homeomorphic to a $k_q$-euclidean ball, where $k_q$ is the dimension of {the subspace $L_q \subset T_pM$ generated by $\theta$}.
\end{definition}
\begin{remark}
We demand that $B_{\eta}^g(q)$ is homeomorphic to a ball in order to avoid ``infinitesimal" self-returns which would not be consistent with definition \ref{def:gFB}. More precisely, suppose that $\theta$ is one dimensional and that there exists a sequence of periodic orbits converging to $p$ with period converging to zero, that is, there exists sequences $\eta_i \to 0$ and $q_i \to p$ such that $B_{\eta_i}^g(q_i)$ is isomorphic to a circle. This would contradict definition \ref{def:gFB}, but would not contradict  
definition \ref{def:gFBM} if we did not demand that $B_{\eta}^g(q)$ is homeomorphic to a ball.
\end{remark}
\noindent
Such a definition coincides with the Flow-Box property whenever $(M,\theta,X)$ is transverse, but also extends to the singular locus of $\theta$ and $X$. Just as for line foliations, we consider the following property:

\begin{definition}
A triple $(M,\theta,X)$ is \emph{geometrically quasi-transverse} if at every point $p$ in $X$, the {Zariski} tangent space of $X$ at $p$ has trivial intersection with the subspace of $T_pM$ generated by $\theta$. 
\end{definition}

Under this condition, we can prove the following result:
\begin{theorem}
Let $(M,\theta,X)$ be a geometrically quasi-transverse triple and assume that $X$ is a non-singular curve, then $(M,\theta,X)$ satisfies the generalized Flow-Box property. Nevertheless, whenever $dim M\geq 4$, there exists a geometrically quasi-transverse triple $(M,\theta,X)$ (where $X$ is non-singular) that does not satisfy the {generalized} Flow-Box property.
\label{th:Rdsimpl}
\end{theorem}
\noindent
We prove this result in {section \ref{ssec:mainTh} as a consequence of} theorem \ref{th:Rd}. In subsection \ref{ssec:AQT}, we introduce a notion of quasi-transversality based on the stability under certain operations, including blowings-up. We give a precise definition (definition \ref{def:AQT}) and statement of the result (Theorem \ref{th:Rd}) on section \ref{sec:AnyDim}. 

%\begin{remark}
%We believe that a geometrically quasi-transverse triple $(M,\theta,X)$ where $X$ is non-singular and $dim X +dim \theta = dim M$, also satisfies the generalized Flow-Box property.
%\end{remark}
%
\subsection{Structure and notations of the manuscript}
\label{ssec:Structure}
The main technique used in the manuscript is resolution of singularities subordinated to the singular distribution $\theta$. We introduce the main technical tool (resolution by $\theta$-invariant centers) in section \ref{sec:RS} and we rely on \cite{BelJA,BeloT} for the proofs. In section \ref{sec:OneDim} we prove the main results for vector-fields, and in section \ref{sec:AnyDim} we provide the proofs for general foliations.

Since we use resolution of singularities arguments, it is convenient to work with ideal sheaves and to keep track of divisors. We consider a quadruple $(M,\theta,\mathcal{I},E)$, which we call \textit{foliated ideal sheaf}, where 
\begin{itemize}
\item $\mathcal{I}$ is a reduced ideal sheaf (whose support $V(\mathcal{I})$ plays the role of the sub-variety $X$);
\item $E$ is a simple normal crossing divisor on $M$ (i.e. {$E$ is the union of smooth connected hypersurfaces and, at each point $p\in E$}, there exists a coordinate system $\pmb{x} = (x_1,\ldots, x_m)$ centered at $p$ such that $E$ is locally given by $\{\Pi_{i=1}^r x_i = 0\}$ for some $r\leq m$);
\item {$\theta$ is coherent sub-sheaf of $Der_M(-\log E)$ closed under the Lie-bracket operation. In particular, $\theta$ is tangent to $E$ (i.e. its local sections $\partial$ are everywhere tangent to $E$).} 
\end{itemize}
If $\theta$ is a $1$-dimensional free singular distribution (i.e. at each point $p \in M$, $\theta$ is locally generated by a vector-field $\partial$), we say that $(M,\theta,\mathcal{I},E)$ is a $1$-foliated ideal sheaf. In what follows, we re-introduce the definitions of quasi-transversality and generalized Flow-Box property for foliated ideal sheaves.

\subsection{Origin of the Problem}
\label{ssec:Or}
Originally, this work aimed to prove that the semi-universal equisingular unfolding of one-forms constructed by J.-F. Mattei in \cite[Theorem (3.2.1)]{Mattei}, is actually universal. This is an important problem in the theory of families of one-forms and foliations. Mattei's construction leads to a local problem involving the action of an algebraic group over an analytic sub-variety. This problem was brought to our attention by J.-F. Mattei in a private communication and stands as the main technical obstacle to prove that the semi-universal unfolding is actually universal. More precisely, the original question can be stated as follows:

\begin{problem}
Consider an (not necessarily compact) algebraic group $G$ acting over a regular manifold $M$ and an analytic and regular sub-variety $X$ of $M$. Suppose that at every point $p$ in $X$, the orbit of $G$ passing through $p$ is ``quasi-transverse" to $X$, i.e. the intersection of the tangent space of the orbit with the tangent space of $X$ is trivial. Is it possible to find a neighbourhood $U \subset G$ of the neutral element and a neighbourhood $W$ of $p$ such that, for every $q$ in $W$ and every $g$ in $U$, $g(q) \in X \implies g(q)=q$?
\end{problem}

Stating this problem for the Lie algebra instead of the Lie group yields the definitions of quasi-transversality and generalized Flow-Box introduced here and to the question: does geometrical quasi-transversality implies the Flow-Box property? Although we give a negative answer in Theorem \ref{th:R2simpl}, the original question of J.-F. Mattei is still open, since it concerns an specific algebraic Lie group. We do show, nevertheless, that an approach based on the Lie Algebra is insufficient since, in this context, the answer is negative.

Based on these considerations, we could also introduce a notion of local generalized slices for (not necessarily compact) Lie groups. We recall that the existence of slices for compact Lie group actions is an important tool for the global study of Lie group actions. Under some additional conditions, it is also possible to obtain global slices over non-compact Lie groups (see \cite{Pal}). In any case, up to our knowledge, no local notion of section is given for general non-compact groups; the generalized Flow-Box property might be a natural way to do so.
\section*{Acknowledgments}
We would like to thank Jean-Fran\c cois Mattei for bringing the problem to our attention and for the very useful discussions on the topic. The second author  was supported by ANR-16-CE40-0008. %{We would also like to thank the two anonymous referees for the many useful suggestions, in particular concerning Definition \ref{def:AQT}}.

\section{Background: Resolution of singularities}
\label{sec:RS}

We briefly recall some basic notations and ideas of resolution of singularities. {Let $E \subset M$ be a simple normal crossing divisor (that is, $E$ is the union of smooth hypersurfaces and, at each point $p\in E$, there exists a coordinate system $x= (x_1,\ldots,x_n)$ such that $E=\{\Pi_{i=1}^m x_i^{\epsilon_i} = 0\}$, where $\epsilon_i \in \{0,1\}$).} {Given a blowing-up $\sigma: \widetilde{M} \to M$, denote by $F $ its exceptional divisor. We say that the blowing-up is \emph{admissible} in respect to a divisor $E \subset M$ if the center $\mathcal{C}$ is a closed and regular sub-manifold of $M$ which has SNC with $E$ (that is, there exists a coordinate system $\pmb{x}= (x_1,\ldots, x_m)$ where the reduced ideal sheaf whose support is $\mathcal{C}$, which we denote by $\mathcal{I}_{\mathcal{C}}$, is locally given by $\mathcal{I}_{\mathcal{C}} = (x_1, \ldots, x_t)$ and $E = \{\Pi_{i=1}^m x_i^{\epsilon_i} = 0\}$, where $\epsilon_i \in \{0,1\}$). If the blowing-up is admissible, we denote it by $\sigma: (\widetilde{M},\widetilde{E})\to (M,E)$, where $\widetilde{E} = \sigma^{-1}(E) \cup F$.}

{Consider an ideal sheaf $\mathcal{I}$}. If {the center of blowing-up} $\mathcal{C}$ is contained in the support of $\mathcal{I}$, we say that $\sigma$ is a blowing-up with \emph{order one} in respect to $\mathcal{I}$. In this case we denote the blowing-up by $\sigma: (\widetilde{M},\widetilde{\mathcal{I}},\widetilde{E})\to (M,\mathcal{I},E)$, where $\widetilde{\mathcal{I}}$ is the weak transform of $\mathcal{I}$, that is $\mathcal{I}_F \widetilde{\mathcal{I}} = \sigma^{\ast}\mathcal{I}$, where $\mathcal{I}_F$ is the reduced ideal sheaf whose support is {the exceptional divisor} $F$ {of $\sigma$ (c.f. \cite[Definition 2.2]{BM})}. {We now consider an involutive singular distribution $\theta$ which is tangent to $E$.}

\begin{remark}\label{rk:thetaAction}
Since $\theta$ is a coherent sub-sheaf of the derivations over $M$, it naturally acts on the structure sheaf $\mathcal{O}_M$ and, thus, over ideal sheaves $\mathcal{I}$. More precisely, given a point $p \in M$, the ideal $\theta[\mathcal{I}]\cdot \mathcal{O}_p$ is defined as the ideal generated by 
\[
\{\partial(f); f \in \mathcal{I}\cdot \mathcal{O}_p,\, \partial \in \theta \cdot \mathcal{O}_p\}
\]
where $\partial$ is a local generator of $\theta$. We denote by $\theta[\mathcal{I}]$ the ideal sheaf whose stalks are defined as above. Since $\theta[\mathcal{I}]$ is again an ideal sheaf, we can iterate the process a finite number of times, say $k$, {and obtain an ideal} which we denote by $\theta^{k}[\mathcal{I}]$.
\end{remark}

\begin{definition}[$\theta$-invariant blowing-up]
We say that a blowing-up $\sigma$ is $\theta$\emph{-invariant}, if the center of blowing-up is $\theta$-invariant, that is, if $\theta[\mathcal{I}_{\mathcal{C}}] \subset \mathcal{I}_{\mathcal{C}}$, where $\mathcal{I}_{\mathcal{C}}$ is the reduced ideal sheaf whose support is $\mathcal{C}$.
\end{definition}
\begin{remark}\label{rk:InvReg}
\begin{itemize}
\item[(1)]
Given a $\theta$-invariant blowing-up of order one $\sigma: (\widetilde{M},\widetilde{\mathcal{I}},\widetilde{E})\to (M,\mathcal{I},E)$, the total transform $\widetilde{\theta} = \sigma^{\ast}(\theta)$ is well-defined{,} analytic {and tangent to $\widetilde{E}$}. {Indeed, f}ollowing the notation introduced in Remark \ref{rk:thetaAction}, {we note} that:
\[
{\theta^{\ast}[\mathcal{I}_F] = \sigma^{\ast}(\theta[\mathcal{I}_{\mathcal{C}}]) \subset \sigma^{\ast}(\mathcal{I}_{\mathcal{C}}) = \mathcal{I}_F}
\]
where $F$ is the exceptional divisor of $\sigma$. {Therefore, we can denote the blowing-up by}
\[
{\sigma : (\widetilde{M},\widetilde{\theta},\widetilde{\mathcal{I}},\widetilde{E})\to (M,\theta,\mathcal{I},E)}
\]
\item[(2)] {Furthermore, by using \cite[Lemma 3.2.i]{BelJA} and part (1), we get that}: 
\[
\mathcal{I}_F \cdot \left( \sum_{i=0}^{\nu}\widetilde{\theta}^i[\widetilde{\mathcal{I}}]\right)  = \sigma^{\ast}\left(\sum_{i=0}^{\nu}\theta^i[\mathcal{I}]\right), \quad \forall \,\nu\in \mathbb{N}
\]  
\end{itemize}
\end{remark}
\begin{proposition}
Let $(M,\theta, \mathcal{I},E)$ be a foliated ideal sheaf (the leaf-dimension of $\theta$ is not necessarily one). Consider a relatively compact open set $M_0$ of $M$ and {denote by $\mathcal{I}_0 = \mathcal{I}|_{M_0}$ and $E_0 = E|_{M_0}$. L}et $\nu$ be a natural number such that $\theta^{\nu+1}(\mathcal{I}) \subset \sum_{i=0}^{\nu}\theta^i(\mathcal{I})$ in $M_0$. Then, there exists a sequence of $\theta$-invariant blowings-up of order one:
\[
\begin{tikzpicture}
  \matrix (m) [matrix of math nodes,row sep=3em,column sep=3em,minimum width=1em]
  {(\widetilde{M},{\widetilde{\theta},}\widetilde{\mathcal{I}}, \widetilde{E}) = (M_r,{\theta_r,}\mathcal{I}_r,E_r) & \cdots & (M_0,{\theta_0,}\mathcal{I}_0,E_0)\\};
  \path[-stealth]
    (m-1-1) edge node [above] {$\sigma_r$} (m-1-2)
    (m-1-2) edge node [above] {$\sigma_1$} (m-1-3);
\end{tikzpicture}
\]
such that: i) the sequence of blowings-up is an isomorphism outside of the support of $\sum_{i=0}^{\nu}\theta^i(\mathcal{I})$; ii) the pulled back foliation $\widetilde{\theta}$ is (well-defined{,}) analytic {and tangent to $\widetilde{E}$}. Furthermore $\sum_{i=0}^{\nu}\widetilde{\theta}^i(\widetilde{\mathcal{I}})$ has empty support (that is, $ \sum_{i=0}^{\nu}\widetilde{\theta}^i(\widetilde{\mathcal{I}})=\mathcal{O}_{\widetilde{M}}$).
\label{prop:HSimplificado}
\end{proposition}
\begin{proof}
This result follows from the same argument as \cite[Proposition 5.2]{BelJA} and relies on the functorial property of resolution of singularities \cite{BM}; a complete proof can be found in \cite[Proposition 5.4.1]{BeloT}.
\end{proof}

\begin{remark}\label{rk:HSimplificado}
If $(M,\theta,\mathcal{I},E)$ is geometrically quasi-transverse (see definitions \ref{def:D1} and \ref{def:GQT}), then condition $(i)$ of Proposition \ref{prop:HSimplificado} implies that $\sigma$ is an isomorphism outside of $Sing(\theta) \cup E$. Indeed, the quasi-transversality hypothesis guarantees that all $\theta$-invariant centres must be contained in $Sing(\theta) \cup E$ and its transform. 
\end{remark}

\section{Generalized Flow-Box property for vector-fields}
\label{sec:OneDim}
Following subsection \ref{ssec:Structure}, we work with $1$-foliated ideal sheaves $(M,\theta,\mathcal{I},E)$ instead of triples $(M,\partial,X)$ (where $Supp \, \mathcal{I} = X$ and $\theta$ is generated by $\partial$). We start by extending the definitions given in subsection \ref{ssec:InV}:

\begin{definition}\label{def:D1}
Consider a $1$-foliated ideal sheaf $(M,\theta,\mathcal{I},E)$, the variety $X = Supp \, \mathcal{I}$ and vector-fields $\partial_U$, defined in open sets $U$ (which cover $M$), that generate $\theta$ on $U$. We say that $(M,\theta,\mathcal{I},E)$:
\begin{itemize}
\item is \emph{geometrically quasi-transverse} if all the triples $(U,\partial_U,X\setminus E)$ are geometrically quasi-transverse.
\item is \emph{analytically quasi-transverse} if it is geometrically quasi-transverse and 
\[
\theta^2[\mathcal{I}] \subset \mathcal{I} + \theta[\mathcal{I}].
\]
\item satisfies the generalized \emph{Flow-Box property} if for every compact set $K \subset X$, there exists a real number $\delta>0$ such that for every point $p \in K \setminus (Sing(\partial) \cup E)$ we have that $\varphi_p(t) \notin X$ for all $0<|t|<\delta$, where $\varphi_p(t)$ is the orbit of $\partial_U$, for some $U$, such that $\varphi_p(0) = p$.
\end{itemize}
\end{definition}

\begin{remark}
Given a geometrically quasi-transverse $1$-foliated ideal sheaf $(M,\theta,\mathcal{I},E)$, note that $\theta$ may be tangent to $X \cap E$. 
\end{remark}

We are ready to state the main result of this section:
\begin{theorem}
Let $(M,\theta,\mathcal{I},E)$ be a geometrically quasi-transverse $1$-foliated ideal sheaf. If one of the following conditions is verified:
\begin{itemize}
\item[i)] The $1$-foliated ideal sheaf $(M,\theta,\mathcal{I},E)$ is analytically quasi-transverse;
\item[ii)] The ideal sheaf $\mathcal{I}$ is regular and the variety $X: = Supp \, \mathcal{I}$: has dimension one or co-dimension one; has normal crossings with $E$ {(that is, if $p \in X \cap E$, there exists a local coordinate system $x=(x_1,\ldots,x_n)$ such that $X = V(x_1,\ldots,x_t)$ and $E= \{x_i^{\epsilon_i}=0\}$ where $\epsilon_i \in \{0,1\}$)}; and is not contained in $E$.
\end{itemize}
Then $(M,\theta,\mathcal{I},E)$ satisfies the generalized Flow-Box property.
\label{th:R1}
\end{theorem}

\begin{proof}
Fix a point $p$ in $X = V(\mathcal{I})$ and a vector-field $\partial$ which generates $\theta\cdot \mathcal{O}_p$. By Proposition \ref{prop:HSimplificado} there exists a relatively compact neighbourhood $M_0$ of $p$ and a sequence of blowings-up of order one:
\[
\begin{tikzpicture}
  \matrix (m) [matrix of math nodes,row sep=3em,column sep=3em,minimum width=1em]
  {(\widetilde{M},\widetilde{\mathcal{I}}, \widetilde{E}) = (M_r,\mathcal{I}_r,E_r) & \cdots & (M_0,\mathcal{I}_0,E_0)\\};
  \path[-stealth]
    (m-1-1) edge node [above] {$\sigma_r$} (m-1-2)
    (m-1-2) edge node [above] {$\sigma_1$} (m-1-3);
\end{tikzpicture}
\]
such that: i) the sequence of blowings-up is an isomorphism outside of $E\cup Sing(\partial)$ (see Remark \ref{rk:HSimplificado}; ii) the pulled back vector-field $\widetilde{\partial}$ is analytic and $\sum_{i=0}^{\nu}\widetilde{\partial}^i(\widetilde{\mathcal{I}})=\mathcal{O}_{\widetilde{M}}$. In particular, by $(i)$, the inverse image of all regular orbits of $\partial$ outside of $E$ are regular orbits of $\widetilde{\partial}$ outside of $\widetilde{E}$ (i.e there is no change on the time of the orbits). Thus, since blowings-up are proper morphisms, we just need to verify that $(\widetilde{M}, \widetilde{\theta},\widetilde{\mathcal{I}},\widetilde{E})$ satisfies the generalized Flow-Box property {in a neighborhood of every point in $V(\widetilde{\mathcal{I}})$}. In what follows, we denote by $\widetilde{X}$ the sub-variety $V(\widetilde{\mathcal{I}})$ and we divide in four cases depending on the initial hypothesis:

\medskip\noindent
\emph{Case I - $(M,\theta,\mathcal{I},E)$ is analytically quasi-transverse:} In this case, $\nu \leq 1$, which implies that {$\widetilde{\partial}(\widetilde{\mathcal{I}}_p) = \mathcal{O}_{\widetilde{M},p}$ for any point $p\in\widetilde{X}$}. Thus, the Flow-box theorem is valid at each point of the sub-variety $\widetilde{X}$ and we conclude {by Remark \ref{rk:FLBOX}(2)}.

\medskip\noindent
\emph{Case II - The ideal sheaf $\mathcal{I}$ is regular and the variety $X=V(\mathcal{I})$ has dimension one:} Note that $\widetilde{X} \cap \widetilde{E}$ has dimension zero. Since $ \widetilde{E}$ is invariant by $ \widetilde{\theta}$ (which is {regular over $\widetilde{X}\cap \widetilde{E}$}), we conclude that $ \widetilde{\theta}$ must be transverse to $\widetilde{X}$ also over $\widetilde{X} \cap \widetilde{E}$. Thus, the Flow-box theorem is valid at each point of the sub-variety $\widetilde{X}$ and we conclude {as above}.

\medskip\noindent
\emph{Case III - The ideal sheaf $\mathcal{I}$ is regular, the variety $V(\mathcal{I})$ has co-dimension one and $\mathbb{K} = \mathbb{C}$:} Let $q$ be a point of $\widetilde{X}$, the function $f$ be a local generator of $\widetilde{\mathcal{I}}$ and consider the sub-variety $W:=V(f,\widetilde{\partial}(f))$. Assume, by contradiction, that $\widetilde{\partial}(f)$ is not a unit, i.e that $W \neq \emptyset$. Then, {since $W$ has codimension two and is empty outside of $\widetilde{E}$ (by geometrically quasi-transversality)}, $W$ is a non-empty co-dimension one {sub}variety {of} $\widetilde{E}$ which must be equal to $\widetilde{X}\cap \widetilde{E}$. Since $\widetilde{E}$ is invariant by $\widetilde{\theta}$ we conclude that $W$ is also invariant by $\widetilde{\theta}$. This contradicts the fact that $\widetilde{\partial}^k(f)$ is a unit for some $k \le \nu$. Thus, $\widetilde{\partial}(f)$ must be a unit itself and $ \widetilde{\theta}$ must be transverse to $\widetilde{X}$ at every point. Thus, the Flow-box theorem is valid at each point of the sub-variety $\widetilde{X}$ and we conclude {as above}.

\medskip\noindent
\emph{Case IV - The ideal sheaf $\mathcal{I}$ is regular, the variety $V(\mathcal{I})$ has co-dimension one and $\mathbb{K} = \mathbb{R}$:} Let $q$ be a point of $\widetilde{X}$, the function $f$ be a generator of $\widetilde{\mathcal{I}}$ and consider the sub-variety $W:=V(f,\widetilde{\partial}(f))$. {The same arguments of Case III are valid if $W$ has codimension one in $\widetilde{X}$ or in $\widetilde{E}$, so we may assume that $W$ is a sub-variety of} co-dimension at least two {of} $\widetilde{X}$ and $\widetilde{E}$. Now, since $\widetilde{X}$ is locally orientable, let {$g$ be an analytic Riemannian metric in a neighborhood $U$ of $q$ and set $\mathcal{N} = \nabla_g(f)$ in $U$}. The function:
\[
\begin{array}{cccc}
\phi: & U \cap \widetilde{X} & \longrightarrow & \mathbb{R} \\
 & a & \mapsto & \left<\mathcal{N}, \widetilde{\partial} \right>(a)
\end{array}
\]
is continuous and equal to zero, if and only if, $a$ is point in $W$. Now, since $W$ has codimension at least two in $\widetilde{X}$,  without loss of generality, we can assume that $\phi \geq 0$. This implies that the order of tangency between $\tilde{\partial}$ and $\tilde{X}$ is odd. Therefore the regular orbits of ${\tilde{\partial}}$ always cross $\widetilde{X}$ from the negative to the positive side (with respect to the orientation defined by $\mathcal{N}$). We conclude {the generalized Flow-Box property in $U$} from this observation.
\end{proof}
\subsection{Proof of Theorem \ref{th:R2simpl}(b)}
\label{ssec:Example}
We present an example of a geometrically quasi-transverse triple $(M,\partial,X)$ with $dim M = 4$ that \textbf{does not} satisfy the generalized Flox-Box condition. We remark that this example is valid for a four dimension real or complex manifold $M$, but it can be generalized to any {higher} dimension. Its construction relies in the following preliminary result:
\begin{lemma}
The equation:
\begin{equation}
(1+s^2)\cos(\alpha)\sin(\alpha) - \alpha = 0
\label{eq:ce}
\end{equation}
has an analytic solution $(\alpha,s) = (h(s),s) = (s U(s),s)$ defined in a neighbourhood of the origin, where $U(0) =\sqrt{6} / 2$.
\label{lem:ce1}
\end{lemma}
\begin{proof}
Not{e} that the equation:
\begin{equation}
(1+s^2)cos(\alpha)\frac{sin(\alpha)}{\alpha} - 1 = 0
\label{eq:ce2}
\end{equation}
has the same solutions of equation $(\ref{eq:ce})$ apart from $(\alpha,s) = (0,s)$. Now, taking the Taylor expansion with respect to the variable $\alpha$ of equation $(\ref{eq:ce2})$ at $(\alpha,s) = (0,0)$, we get
\[
s^2 -\frac{2}{3}(1+s^2)\alpha^2 + o(\alpha^3) = 0
\]
so, by the Weierstrass Preparation Theorem and the symmetry of equation $(\ref{eq:ce2})$ with respect to the transform $\alpha \to -\alpha$, the equation $(\ref{eq:ce2})$ can be written as:
\[
( f(s)^2 {-}\alpha^2)u(\alpha,s) = 0
\]
where $u(\alpha,s)$ is an unit and $f(s)$ is an analytic function. Furthermore, we have that $f(s)^2 u(0,s) = s^2$, which implies that $f(s) = \pm s U(s)$, where $U(s)$ is an analytic unit defined in a neighbourhood of zero. Taking $h(s) = s U(s)$ gives the desired result. To finish, we just need to use the Taylor expansion of equation $(\ref{eq:ce})$ in respect to $s$ to conclude that $U(0) = \sqrt{6}/ 2$.
\end{proof}
Now, let $(x,y,z,w)$ be a globally defined coordinate system of $\mathbb{K}^4$ (where $\mathbb{K}=\mathbb{R}$ or $\mathbb{C}$). We consider the triple $(M,\partial,X)$ where the manifold $M$ is a (sufficiently) small open neighbourhood of the origin of $\mathbb{K}^4$ (where $h(x^2+y^2)$ is well-defined); the vector-field $\partial$ is given by:
\[
\partial = y^2 \frac{\partial}{\partial w} + x \frac{\partial}{\partial y} - y\frac{\partial}{\partial x}
\]
and the (non-singular) variety $X$ is given by $z= f(x,y)$ and $w=g(x,y)$ where:
\[
\begin{aligned}
f(x,y) &= y^2 \cos (h(x^2+y^2)) - xy \sin (h(x^2+y^2))\\
g(x,y) &= \frac{1}{2}xy(x^2+y^2)^2
\end{aligned}
\]

Theorem \ref{th:R2simpl}(b) follows from the two Claims below:

\medskip\noindent
\textbf{Claim 1:} In this example, $(M,\partial,X)$ is geometrically quasi-transverse.
\begin{proof}
In order to prove the Claim, we just need to show that the {set $W$ of} points where $\partial$ and $X$ are tangent is contained in $Sing(\partial) \cap X = \{0,0,0,0\}$. Indeed, {$W$} is a variety {defined} by $z= f(x,y)$, $w=g(x,y)$, and
\[
\begin{aligned}
\phantom{2\, }\partial[z - f(x,y)] &={-}2xy \cos (h(x^2+y^2)) + (x^2-y^2) \sin (h(x^2+y^2)) =0\\
2\, \partial[w - g(x,y)] &= 2 y^2 -(x^2-y^2)(x^2+y^2)^2 =0
\end{aligned}
\]
Now, from the forth equation we get two solutions in $y$ (defined in a sufficiently small neighbourhood of the origin):
\[
y_1=x^3 V(x) \quad \text{and} \quad
y_2= - x^3 V(x)
\]
where $V(x)$ is an analytic function such that $V(0)= \frac{1}{2} \sqrt{2}$. Now, substituting in the third equation and taking its Taylor expansion in $x$, we get:
\[
\left[ \pm 2V(0)+U(0)\right]x^4 + O(x^5) = 0
\]
where we recall that $h(s) = s U(s)$. Since $\pm 2V(0)+U(0)\neq 0$, we conclude that $W=\{0,0,0,0\}$ as we wanted to prove.
\end{proof}
\noindent
\textbf{Claim 2:} In this example, $(M,\partial,X)$ does not satisfy the generalized Flow-Box property.
\begin{proof}
We prove this claim when $\mathbb{K}=\mathbb{R}$ since the complex case trivially follows from the real case. Since the variety $\mathcal{C} = \{x=0,y=0\}$ is invariant by $\partial$, we consider the (real) blowing-up $\sigma: \widetilde{M} \to M$ with center $\mathcal{C}$ and exceptional divisor $E$, where $\widetilde{M} =\mathbb{R}^{+} \times S^1 \times \mathbb{R}^2$ and there exists a globally defined coordinate system $(r,\alpha,z,w)$ on $\widetilde{M}$ where $\sigma$ is given by:
\[
x= r \cos(\alpha), \quad y=r \sin(\alpha).
\]
The pulled-back vector-field $\widetilde{\partial}$ is equal to:
\[
\widetilde{\partial} = r^2 \sin(\alpha)^2 \frac{\partial}{\partial w} + \frac{\partial}{\partial \alpha}
\]
and the pulled-back variety $\widetilde{X}$ is {defined} by $z = \tilde{f}(r,\alpha)$ and $w = \tilde{g}(r,\alpha))$, where:
\[
\begin{aligned}
\tilde{f}(r,\alpha) &= r^2 \sin(\alpha) [ \sin(\alpha) \cos (h(r^2)) -  \cos(\alpha) \sin (h(r^2)) ]\\
\tilde{g}(r,\alpha) &= \frac{1}{2} r^6\sin(\alpha)\cos(\alpha)
\end{aligned}
\]
For a fixed $r_0>0$, the orbit of $\widetilde{\partial}$ passing by $(r,\alpha,z,w) = (r_0,0,0,0)$ at $t=0$ is given by:
\[
\varphi(r_0,t)= \left(r_0,t,0,\frac{r_0^2}{4} [ 2t - \sin(2t)]\right)
\]
\begin{figure}
  \centering
\includegraphics[scale=0.7]{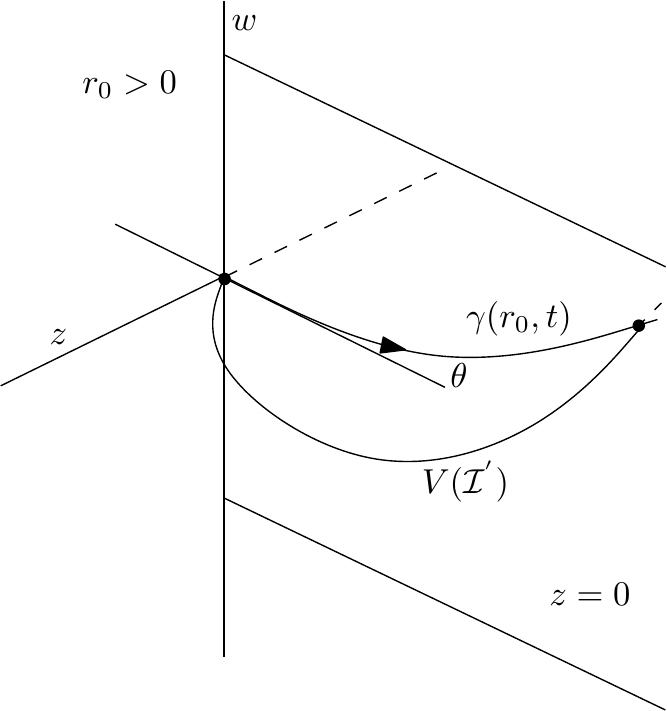}
\caption{\small{In the picture $r_0>0$ is fixed. The orbit $\varphi(r_0,t)$ is contained in the plane $\{z=0\}$ and the variety $\widetilde{X}$ is a curve that cuts the orbit $\varphi(r_0,t)$ two times.}}
\label{fig:CE}
\end{figure}
\noindent
Now, at one hand:
\[
\begin{aligned}
\left[z - \tilde{f}(r,\alpha)\right] (r_0,0,0,0) &= r_0^2 \sin(0) [ \sin(0 - h(r_0^2)) ] = 0\\
\left[w - \tilde{g}(r,\alpha)\right] (r_0,0,0,0) &= \frac{1}{2}r_0^6\sin(0)\cos(0) = 0
\end{aligned}
\]
and $(r_0,0,0,0) \in \widetilde{X}$. At the other hand:
\[
\begin{aligned}
\left[z - \tilde{f}(r,\alpha)\right] (\varphi(r_0,h(r^2_0))) &= r_0^2 \sin(h(r^2_0)) [ \sin(h(r^2_0) - h(r^2)) ] = 0 \\
2\left[w - \tilde{g}(r,\alpha)\right](\varphi(r_0,h(r^2_0))) &=  \frac{r_0^2}{2} \left[ 2 h(r_0^2) - \sin(2 h(r_0^2))\right] - r_0^6\sin(h(r^2_0))\cos(h(r^2_0))
\end{aligned}
\]
and, re-ordering the terms of the second equation, we get:
\[
r_0^2\left[h(r_0^2) - \sin( h(r_0^2)) \cos( h(r_0^2)) (1+r_0^4)\right]
\]
and, since ${h}(s)$ is a solution of equation \eqref{eq:ce}, we conclude that $\left[w - \tilde{g}(r,\alpha)\right]\allowbreak (\varphi(r_0,h(r^2_0))) =0$, which implies that $\varphi(r_0,h(r^2_0)) \in \widetilde{X}$.

So, for each $r_0>0$ fixed, there exists an orbit of $\widetilde{\partial}$ that cuts $\widetilde{X}$ two times. Furthermore, the time between each cut is equal to $h(r_0^2)$, which goes to zero when $r_0$ goes to zero. Now, the center of blowing-up $\mathcal{C}$ was invariant by the vector field $\partial$ (thus, the time of the orbits is not affected) and $\sigma$ is an isomorphism outside of the exceptional divisor, we conclude that $(M,\partial,X)$ does not satisfy the generalized Flow-Box property.
\end{proof}
\section{Generalized Flow-Box property for singular foliations}\label{sec:AnyDim}
\subsection{Background on sub-Riemannian Geometry}
\label{ssec:SRG}
We follow \cite{Bela}. Consider $\mathbb{K}=\mathbb{R}$, a regular analytic manifold $W$ and globally defined analytic vector fields $\{\partial_1,...,\partial_m\}$. For each point $p$ in $W$, we denote by $L_p$ the subspace of $T_pW$ generated by $\{\partial_1(p),\ldots,\partial_m(p)\}$. Given any vector $v$ of $L_p$, there always exists $(u_1,...,u_m) \in \mathbb{K}$ (not necessarily unique) such that $ v= \sum^m_{i=1} u_i \partial_i(p)$. So, for each point $p$ of $W$, consider the mapping:
\[
\begin{array}{cccc}
\Phi_p: & \mathbb{R}^m & \longrightarrow & T_pW \\
 & (u_1,...,u_m)& \mapsto & \sum^m_{i=1} u_i \partial_i(p)
\end{array}
\]
{Note} that $\Phi_p$ restricted to the linear subspace $(ker \Phi_p)^{\bot}$ is a linear isomorphism onto $L_p$. Let $\Psi_p: L_p \to (ker \Phi_p)^{\bot}$ be the inverse mapping. Then, if $v$ and $w$ are vectors contained in $L_p$, we define the \textit{sub-Riemannian metric} $g_p(v,w)$ associated to $\{\partial_1,...,\partial_m\}$ by:
\[
g_p(v,w) := \left<\Psi_p(v),\Psi_p(w)\right> 
\]
Based on this metric, we define the notion of \textit{sub-Riemannian norms} $\|.\|_{{g},p} = |g_p(\cdot,\cdot)|$ and $\|.\|_{{g},p,\infty} = \max_{i}\{|u_i|\}$ (where $u = \Psi_p(\cdot)$) {generated by} $\{\partial_1,\ldots,\partial_m\}$. We extend both norms to every vector $v$ of $T_pW$ by setting $\|v\|_{g,p} =\|v\|_{g,p,\infty}= \infty$ if $v$ is not contained in $L_p$. {Finally, given two point $p$ and $q$, we defined their sub-Riemannian distance by}
\[
{d_g(p,q) = \min\left\{\int_0^1\|\dot{\gamma}(t)\|_{g,\gamma(t)}dt;\, \gamma \text{ abs. continuous curve s.t. }\gamma(0)=p,\,\gamma(1)=q\right\}}
\]

\begin{remark}
The above definitions can be extended to the complex case in a standard way: if $\{\partial_1,...,\partial_m\}$ are globally defined holomorphic vector fields, let $\{\tilde{\partial}_1, \ldots, \tilde{\partial}_{2m}\}$ be globally defined real-analytic vector fields such that $\{\tilde{\partial}_i,\tilde{\partial}_{m+i}\}$ generates $\partial_i$. We then extend the definitions of sub-Riemmanian metric and norm by using $\{\tilde{\partial}_1, \ldots, \tilde{\partial}_{2m}\}$. 
\end{remark}
Finally, an analytic \textit{sub-Riemannian metric} on $M$ is a function $g: TM \otimes TM \longrightarrow \mathbb{R} \cup \{\infty\}$ which, locally may be defined as the metric associated to some system of analytic vector fields. In other words, for each point $p$ in $M$, there exists a neighbourhood $U_p$ of $p$ and a set of vector fields $\{\partial_1,...,\partial_m\}$ defined over $U_p$ such that $g|_{U_p} = g_{\{\partial_1,...,\partial_m\}}$ where $g_{\{\partial_1,...,\partial_m\}}$ is the sub-Riemmanin metric on $U_p$ {generated by} $\{\partial_1,...,\partial_m\}$.

\subsection{Geometric quasi-transversality}

\begin{definition}\label{def:GQT}
A foliated ideal sheaf $(M,\theta,\mathcal{I},E)$ is said to be \emph{geometrically quasi-transverse} if for {any} point $p$ in $ X \setminus E$ {we have}
\[
dim_{\mathbb{K}} L_p + dim_{\mathbb{K}}T_p X = dim_{\mathbb{K}} ( L_p+ T_p X ) 
\]
where $L_p$ is the linear sub-space of $T_pM$ generated by $\theta$ {and $T_pX$ is the Zariski tangent space of $X$ at $p$}.
\end{definition}

\begin{proposition}
Let $(M,\theta,\mathcal{I},E)$ be a geometrically quasi-transverse foliated ideal sheaf and suppose that $\theta[\mathcal{I}] = \mathcal{O}_M$. Then, at each point $p\in V(\mathcal{I})$ there exists two involutive singular distributions $\theta_Y$ and $\theta_Z$ (defined in some {open neighborhood} $U$ of {$p$}), where: $\theta_Y$ is generated by a vector-field $\partial_Y$, and $\theta_Z$ is generated by vector-fields $\{\partial_{Z_1}, \ldots, \partial_{Z_r}\}$ {and has leaf dimension strictly smaller than $\theta$}, such that: i) $\theta = \theta_Y+\theta_Z$; ii) $\partial_Y(\mathcal{I}\cdot \mathcal{O}_p) = \mathcal{O}_p$; iii) $[\partial_Y,\partial_{Z_i}] \equiv 0$ for all $i\leq r$.
%In particular, $(U,\theta_Z,\mathcal{I}|_U,E|_U)$ is geometrically quasi-transverse (that is, the category $GQT$ satisfies the above leaf-dimension compatibility condition).
%Furthermore, if $(U,\theta_Z,\mathcal{I} \cdot \mathcal{O}_U,E\cap U)$ satisfies the generalized Flow-Box property, so does $(U,\theta,\mathcal{I} \cdot \mathcal{O}_U,E\cap U)$.
\label{prop:LT} 
\end{proposition}
\begin{proof}
In what follows, we use a well-known argument based on a linear ODE; we learned about this argument through \cite{Mit} (where $\theta$ is assumed to be {torsion free}). 

Fix a point $p\in M$ a let $(y,z) = (y,z_{1},...,z_{n-1})$ be a coordinate system centered at $p$ and defined in a neighbourhood $U$ of $p$, such that the vector field $Y = \partial_y$ is contained in $\theta \cdot \mathcal{O}_U$ and the analytic function $y$ is contained in the ideal $\mathcal{I} \cdot \mathcal{O}_U$. Let $\{W_1,...,W_s\}$ be vector fields such that $\{Y, W_1,...,W_s\}$ generate $\theta\cdot \mathcal{O}_U$. We can suppose that $W_i(x) \equiv 0$, which implies that:
\[
[Y,W_j] = \sum^s_{j=1} A_{i,j}(y,z) W_j
\]
Now, consider a vector field of the form $W=\sum^s_{i=1} \mu_i W_i$, where $\mu_i \in \mathcal{O}_U$. Then:
\[
[Y,W]= \sum^s_{j=1} \left[ Y(\mu_j) +  \sum^s_{i=1} \mu_i A_{i,j}(y,z) \right]\, W_j 
\]
Since $Y$ is a regular vector field, the system of equations:
\[
Y(\mu_j) +  \sum^s_{i=1} \mu_i A_{i,j}(y,z) = 0, \quad j=1,\ldots,s
\]
is an analytic system of ODE's. Therefore, up to shrinking $U$, there exists $s$ locally defined analytic solutions $\vec{\mu}_i = (\mu_{i,1},...,\mu_{i,s})$ such that $\vec{\mu}_i(0) = e_i = (0,...,0,1,0,...,0)$ (where the $1$ is on the $i$ position). Without loss of generality, we suppose that these solutions are defined on $U$ and we denote by $Z_i := \sum^s_{j=1} \mu_{i,j} W_j$ the vector field associated to it. Note that:
\begin{itemize}
\item[$\bullet$] The vector fields $\{Y,Z_1,...,Z_s\}$ generates $\theta\cdot \mathcal{O}_U$;
\item[$\bullet$] The Lie bracket $[Y,Z_i] \equiv 0$ for $i =1,\ldots, s$;
\item[$\bullet$] The vector fields $\{Z_1,...,Z_s\}$ generate an involutive $d-1$-singular distribution $\theta_Z$ (because of the two points above).
\end{itemize}
\end{proof}

\begin{lemma}\label{lem:GQTcategory}
Let $(M,\theta,\mathcal{I},E)$ be a geometrically quasi-transverse foliated ideal sheaf and consider a foliated ideal sheaf $(\widetilde{M},\widetilde{\theta},\widetilde{\mathcal{I}}, \widetilde{E})$ defined by one of the following three operations:
\begin{itemize}
\item A $\theta$-invariant blowing-up of order one, that is, $\sigma: (\widetilde{M},\widetilde{\theta},\widetilde{\mathcal{I}}, \widetilde{E}) \to (M,\theta,\mathcal{I},E)$ is a $\theta$-invariant blowing-up of order one;
\item An open restriction, that is, given an open set $U\subset M$, we consider $(\widetilde{M},\widetilde{\theta},\widetilde{\mathcal{I}}, \widetilde{E}) = (U,\theta|_U,\mathcal{I}|_U,E|_U)$.
\item A leaf-reduction operation, that is, suppose that $\theta[\mathcal{I}] = \mathcal{O}_M$. In the notation of Proposition \ref{prop:LT}, we consider $(\widetilde{M},\widetilde{\theta},\widetilde{\mathcal{I}}, \widetilde{E}) = (U,\theta_Z,\mathcal{I}|_U,E_U)$.
\end{itemize}
Then $(\widetilde{M},\widetilde{\theta},\widetilde{\mathcal{I}}, \widetilde{E})$ is also geometrically quasi-transverse.
\end{lemma}
\begin{proof}
The proof follows directly from the definition of quasi-transversality and the fact that a blowing-up is an isomorphism outside its center.
\end{proof}

\begin{remark}
In what follows, we consider the category, which we denote by $GQT$, whose objects are geometrically quasi-transverse foliated ideal sheaves and the morphisms are (composition of) $\theta$-invariant blowings-up of order one, restrictions and leaf-reduction operation (see Lemma \ref{lem:GQTcategory}).
\label{rk:GQT}
\end{remark}
%
%The following result, which concerns a decomposition of $\theta$ in a special case, will guide results and definitions which appear later on (c.f. Proposition \ref{prop:LT2} and Definition \ref{def:AQT}(II)).

\subsection{The generalized Flow-Box property}
\begin{definition}
We say that a foliated ideal sheaf $(M,\theta,\mathcal{I},E)$ satisfies the \emph{generalized Flow-Box property at a point} {$p \in X = V(\mathcal{I})$}, if there exists a triple $(U,\delta,g)$, where $U$ is an open neighbourhood of $p$, $\delta>0$ is a real number and $g$ is a sub-Riemannian metric over $U$ generated by a set $\{\partial_1, \ldots, \partial_d\}$ of local generators of $\theta$ such that: for every point $q \in U \cap (X \setminus  E)$, every $g$-ball of radius $\eta<\delta$
\[
B_{\eta}^g(q) = \{a \in U ; d_g(a,q) <\eta\}
\]
intercepts $X$ only at $q$ and is homeomorphic to a $k_q$-euclidean ball, where $k_q$ is the dimension of the leaf of $\theta$ passing through $q$. {Finally, we say that  $(M,\theta,\mathcal{I},E)$. satisfies the generalized Flow-Box property if it is so at every point $p\in X$.}
\end{definition}
\begin{lemma}
The {generalized} Flox-Box property is independent of the choice of local generators $\{\partial_1, \ldots, \partial_d\}$ of $\theta\cdot\mathcal{O}_p$ (and, thus, it is independent of the local sub-Riemannian metric).
\label{lem:G-FB1}
\end{lemma}
\begin{proof}
Suppose that $(M,\theta,\mathcal{I},E)$ satisfies the {generalized} Flox-Box property at a point $p$ with the sub-Riemannian metric $g$ generated by $\{\partial_1, \ldots, \partial_r\}$ and let $\tilde{g}$ be a sub-Riemmanian metric generated by another set $\{\tilde{\partial}_1, \ldots, \tilde{\partial}_s\}$ of generators of $\theta \cdot \mathcal{O}_p$. Since these are two sets of generators of $\theta \cdot\mathcal{O}_p$, there exists a matrix $A = [a_{i,j}]_{r \times s}$ of analytic functions such that $(\partial_1, \ldots, \partial_r) = A (\tilde{\partial}_1, \ldots, \tilde{\partial}_s)$. Let $\|A\|_p$  be the $\infty$-norm of the matrix $A$ at $p$, i.e. $\|A\|_p= max\{\|a_{i,j}(p)\| \}$. Now, given a vector $\vec{v}$ in $L_p$, we can write it as $\vec{v} = \sum_{i=1}^r u_i \partial_i(p)$ which implies that:
\[
\vec{v} = \sum_i^{r} u_i \sum^s_{j=1} a_{i,j} \tilde{\partial}_j(p) = \sum^s_{j=1} \tilde{\partial}_j(p) \left[\sum^r_{i=1}a_{i,j} u_i\right].
\]
This implies that $\max\{ \|\sum^r_{i=1}a_{i,j} u_i\|\} \leq \|A\|_p \max\{\|u_i\|\}$. Thus: 
\[
\| \vec{v} \|_{\tilde{g},p,\infty} \leq \|A\|_p \| \vec{v} \|_{g,p,\infty}.
\]
So, {up to} taking a smaller open set $U_p$, there exists ${C>0}$ such that $\|\vec{v}\|_{\tilde{g},q} \leq {C}\|\vec{v}\|_{g,q}$ for every $\vec{v}$ in $L_q$ {and every} $q\in U_p$. Furthermore, by the symmetric argument, we conclude that there exists ${c>0}$ such that:
\[
{c}\|\vec{v}\|_{g,q} \leq  \|\vec{v}\|_{\tilde{g},q} \leq {C}\|\vec{v}\|_{g,q}
\]
{Finally, since the metrics are equivalent, up to shrinking the radios $\eta$, the balls $B^{\widetilde{g}}_{\eta}(q)$ are also homeomorphic to an Euclidean balls for every $q\in U_p \cap ( X\setminus E)$}, which {proves the} Lemma.
\end{proof}

\begin{lemma}
[Blowing-up reduction] Let $\sigma: (\widetilde{M},\widetilde{\theta},\widetilde{\mathcal{I}},\widetilde{E}) \to (M,\theta,\mathcal{I},E)$ be a $\theta$-invariant blowing-up of order one and suppose that $(M,\theta,\mathcal{I},E)$ is geometrically quasi-transverse and $(\widetilde{M},\widetilde{\theta},\widetilde{\mathcal{I}},\widetilde{E})$ satisfies the generalized Flow-Box property. Then $(M,\theta,\mathcal{I},E)$ also satisfies the generalized Flow-Box property.
\label{lem:20}
\end{lemma}
\begin{proof}
The result is clear for any point outside of the center of blowing-up. So, let $p$ be a point in the center $\mathcal{C}$, $\{\partial_1, \ldots, \partial_d\}$ be a system of generators of $\theta \cdot\mathcal{O}_p$ and $g$ be the sub-Riemannian metric generated by $\{\partial_1, \ldots, \partial_d\}$. We can assume that $g$ is well-defined in a sufficiently small relatively compact open neighbourhood $U$ of $p$.

Since $(\widetilde{M},\widetilde{\theta},\widetilde{\mathcal{I}},\widetilde{E})$ satisfies the generalized Flow-Box property, at each point $q$ in $\sigma^{-1}(p)$ there exists a triple $(U_q,\delta_q,g_q)$ that satisfies the conditions of the generalized Flow-Box property. So, there exists a finite number $N$ of points $q_i$ on $\widetilde{M}$ such that $\sigma\left (\cup_{i \leq N} U_{q_i} \right)=  U$. We can assume, by lemma \ref{lem:G-FB1}, that $g_{q_i} = \sigma^{\ast}g$, where $\sigma^{\ast}g$ is generated by the pull-back vector-fields $\{\sigma^{\ast}\partial_1, \ldots, \sigma^{\ast}\partial_d\}$. Set $\delta = \min\{\delta_{q_i}; i=1, \ldots, N\}$. We claim that $(U,\delta,g)$ satisfies the conditions of the generalized Flow-Box property at $p$. Indeed, let $q \in (V(\mathcal{I}) \cap U) \setminus E$. If $q \in \mathcal{C}$ then, by the quasi-transversality hypothesis, {we have that $dim(L_q) = 0$}. Otherwise, $q \notin \mathcal{C}$ and $\sigma$ is local isomorphism {at $q$}. We conclude easily from these two observations. 
\end{proof}
\begin{proposition}\label{prop:LT2}
Following the hypothesis and notation of Proposition \ref{prop:LT}, if $(U,\theta_Z,\mathcal{I} |_U,E|_U)$ satisfies the generalized Flow-Box property, so does $(U,\theta,\mathcal{I}|_U ,E|_U)$.
\end{proposition}
\begin{proof}
Assume that $(U,\theta_Z,\mathcal{I},E)$ satisfies the generalized Flow-Box property and let $X = V(\mathcal{I})$. Without loss of generality we can suppose that there exists a coordinate system $(y,\pmb{z}) = (y,z_{1},...,z_{n-1})$ centered at $p$ such that $\partial_Y = \partial_y$ and that $y \in \mathcal{I}$. Now, let $g$ be the sub-Riemannian metric generated by $\{\partial_Y,\partial_{Z_1},\ldots,\partial_{Z_s}\}$, $g_{Y}$ be the sub-Riemannian metric generated by $\{\partial_Y\}$ and $g_{Z}$ be the sub-Riemannian metric generated by $\{\partial_{Z_1},\ldots,\partial_{Z_s}\}$. By the usual Flow-Box Theorem, we know that $(U,\theta_Y,\mathcal{I},E)$ also satisfies the generalized Flow-Box property and, thus, there exists $\delta>0$ such that:
\[
d_{g_{Y}}\left(q, X \setminus \{q\}\right)> \delta ,\quad d_{g_Z}(q, X \setminus \{q\}) > \delta,
\]
%Now, we claim that for all point $q$ in $X \setminus E$, the $g$-ball $B^{g}_{\delta}(q)$ intersects $X=V(\mathcal{I})$ only in $q$. Indeed, we just need to show that $d_{g}(q, X \setminus \{q\}) > \delta$ for all $q \in X\setminus E$.
for all point $q \in X \setminus E$. We claim that $d_{g}(q, X \setminus \{q\}) > \delta$. Indeed, suppose by contradiction that there exists a point $q$ in $X \setminus E$ such that $d_{g}(q, V(\mathcal{I}) \setminus \{q\}) < \delta$. In this case, there exists an absolutely continuous curve $c:[a,b] \longrightarrow U$ such that $c(a)=q$, $c(b) \in X \setminus \{q\}$ and $length_g(c(t))<\delta$. We write $c(t) = (y(t),z(t))$ and we define the absolutely continuous curve $\varphi(t) = (0,z(t))$. Since $y \in \mathcal{I} $, we conclude that $c(a) = (0,z(a))$ and $c(b)=(0,z(b))$ and, thus, $\varphi(a) =q$ and $\varphi(b) \in X \setminus \{q\}$. Now, since $g_Y = dy^2$
\[
\|  
\dot{c(t)} \|_{g, c(t)} \geq  \|\dot{z(t)} \|_{g, c(t)} = \| \dot{z(t)} \|_{g_Z, c(t)}.
\]
Moreover, since $(0,\dot{z(t)}) = \dot{\varphi(t)}$ and $g_Z$ is independent of the $y$ coordinate
\[
\| \dot{z(t)} \|_{g_Z, c(t)} = \| \dot{\varphi(t)} \|_{g_Z, \varphi(t)} \implies \|  \dot{c}(t) \|_{g, c(t)} \geq \| \dot{\varphi(t)} \|_{g_Z, \varphi(t)}.
\]
Which implies that
\[
\mbox{length}_{g_Z}(\varphi(t)) \leq \mbox{length}_{g}(c(t))< \delta,
\]
which is a contradiction. Thus, it only rests to prove that the $g$-ball $B^{g}_{\eta}(q)$, for $\eta <\delta$ is homeomorphic to a $k_q$-euclidean ball, where $k_q$ is the dimension of the leaf of $\theta$ passing through $q$. This is clear in the $\infty$-norm $\|.\|_{g,\infty}$ since in this norm $B^{g,\infty}_{\eta}(q) = B^{g_Y,\infty}_{\eta}(q) \times B^{g_Z,\infty}_{\eta}(q)$. So, the claim for the norm $\|\cdot\|_g$ also follows and we are done.
\end{proof}

%\begin{lemma}[Local Triviality] Let $(M,\theta,\mathcal{I},E)$ be a geometrically quasi-transverse foliated ideal sheaf such that $(M,\theta,\mathcal{I},\emptyset)$ is also geometrically quasi-transverse. Then, $(M,\theta,\mathcal{I},E)$ satisfies the generalized Flow-Box property.
%\label{lem:10}
%\end{lemma}
%\begin{proof}
%Fix a point $p\in X=V(\mathcal{I})$. Apart from using Proposition \ref{prop:LT2}, we may assume that the leaf dimension of $\theta$ passing by $p$ is equal to $0$...
%This result follows from the {Frobenius} Theorem and Proposition \ref{prop:LT2}.
%\end{proof}

\subsection{Analytic quasi-transversality}\label{ssec:AQT}
Following remark \ref{rk:GQT}, we consider the category $GQT$ whose objects are geometrically quasi-transverse foliated ideal sheaves $(M,\theta,\mathcal{I},E)$ and the morphisms are composition of $\theta$-invariant blowings-up of order one, restrictions and leaf-reduction operation (see Lemma \ref{lem:GQTcategory}). In this section, we define a sub-category of $GQT$, which we denote by $AQT$ (and whose objects are said to be analytically quasi-transverse), which satisfies the following condition: an object $(M,\theta,\mathcal{I},E)$ whose leaf dimension is one belongs to $AQT$ if, and only if, it is analytically quasi-transverse according to definition \ref{def:D1} (c.f. Remark \ref{rk:AQT}(2) below). More precisely:%; and $AQT$ is closed by the leaf dimension compatibility condition given in Proposition \ref{prop:LT} (just as the category $GQT$). More precisely:

\begin{definition} \label{def:AQT}
We denote by $AQT$ the maximal sub-category of $GQT$ whose objects $(M,\theta,\mathcal{I},E)$ satisfy the following condition:
\begin{equation}\label{eq:aqt}
\theta^2[\mathcal{I}] \subset \theta[\mathcal{I}] + \mathcal{I}
\end{equation}

%\begin{itemize} 
%
%\item[(*)] The relation $\theta^2[\mathcal{I}] \subset \theta[\mathcal{I}] + \mathcal{I}$ is satisfied;
%
%\item[(ii)] If $\sigma:(\widetilde{M},\widetilde{\theta},\widetilde{\mathcal{I}},\widetilde{E}) \to (M,\theta,\mathcal{I},E)$ is a $\theta$-invariant blowing-up of order one, then $(\widetilde{M},\widetilde{\theta},\widetilde{\mathcal{I}},\widetilde{E})$ is analytically quasi-transverse;
%
%\item[(II)] Given a point $p\in V(\mathcal{I})$, suppose that there exists  two involutive singular distributions $\theta_Y$ and $\theta_Z$ of $\mathcal{O}_p$, where: $\theta_Y$ is generated by a vector-field $\partial_Y$, and $\theta_Z$ is generated by vector-fields $\{\partial_{Z_1}, \ldots, \partial_{Z_r}\}$ {and have generic leaf dimension strictly smaller than $\theta$}, such that: i) $\theta = \theta_Y+\theta_Z$; ii) $\partial_Y(\mathcal{I} \cdot\mathcal{O}_p) = \mathcal{O}_p$; iii) $[\partial_Y,\partial_{Z_i}] \equiv 0$ for all $i\leq r$. Then, for a sufficiently small open neighbourhood $U$ of $p$, the foliated ideal sheaf $(U, \theta_Z, \mathcal{I}|_U,E|_U)$ also belongs to $AQT$.
%\end{itemize}
Foliated ideal sheaves $(M,\theta,\mathcal{I},E)$ which belong to $AQT$ are said to be analytically quasi-transverse.
\end{definition}
\begin{remark}\label{rk:AQT}
\begin{itemize}
\item[(1)]  Note that, if $\mathscr{F}_1$ and $\mathscr{F}_2$ are two sub-categories of $GQT$ whose objects satisfy condition $\eqref{eq:aqt}$, then their union $\mathscr{F} = \mathscr{F}_1 \cup \mathscr{F}_2$ is also a sub-category whose objects satisfy condition $\eqref{eq:aqt}$. This remark, together with Zorn's Lemma, guarantee that $AQT$ is well-defined (that is, that there exists an \emph{unique} maximal sub-category which is closed by condition \eqref{eq:aqt}).

\item[(2)]

Note that $1$-foliated ideal sheaves $(M,\theta,\mathcal{I},E)$ satisfying the conditions of definition \ref{def:D1} belong to $AQT$. Indeed, denote by $aqt$ the subset of foliated ideal sheaves satisfying definition~\ref{def:D1}, and by $AQT(1)$ the set of objects in $AQT$ whose leaf dimension is $1$. It follows from the definition that $AQT(1) \subset aqt$. Next, it is not difficult to see that the union of $aqt$ with all objects whose leaf dimension is zero, forms a subcategory of $GQT$ (that is, closed under the admissible  morphisms, c.f Remark \ref{rk:InvReg}(2)) and it satisfies \eqref{eq:aqt}. Therefore $aqt \subset AQT(1)$ by maximality of $AQT$, which implies that these sets are equal.

\item[(3)] A geometrically quasi-transverse foliated ideal sheaf $(M,\theta, \mathcal{I},E)$ may satisfy condition \eqref{eq:aqt} and \emph{not} satisfy the generalized Flow-Box property. Indeed, we consider a modification of the example given in subsection \ref{ssec:Example}. Let $M = U \times \mathbb{K} \subset \mathbb{K}^5$ (with coordinate system $(x,y,z,w,v)$), where $U$ is a sufficiently small open neighborhood of the origin of $\mathbb{K}^4$ where $h(x^2+y^2)$ is well-defined (see Lemma \ref{lem:ce1}); we consider  $\theta$ generated by
\[
\partial_1 = y^2 \frac{\partial}{\partial w} +x \frac{\partial}{\partial y}- y \frac{\partial}{\partial x},\qquad \partial_2 = \frac{\partial}{\partial v} 
\]
and $\mathcal{I}= (z-f(x,y),w-g(x,y),v)$ (see definition of $f$ and $g$ on subsection \ref{ssec:Example}). It is easy to see that $(M,\theta,\mathcal{I},E=\emptyset)$ satisfies condition \eqref{eq:aqt}. Furthermore, following subsection \ref{ssec:Example}, we can verify that this example is geometrically quasi-transverse, but does not satisfy the generalized Flow-Box property. Note, nevertheless, that $(M,\theta,\mathcal{I},E=\emptyset)$ is \emph{not} algebraically quasi-transverse, since after a leaf-reduction operation, we recover the example of subsection \ref{ssec:Example}.
%be analytically quasi-transverse. Indeed, consider $M= \mathbb{R}^5$, $\theta = (\partial_{x_1}, x_4 \partial_{x_2} +  x_5 \partial_{x_3}+ \partial_{x_5})$, $\mathcal{I} = (x_1,x_2,x_3)$ and $E =(x_4=0)$; note that $(M,\theta,\mathcal{I},E)$ is geometrically quasi-transverse and satisfy condition \eqref{eq:aqt}. Nevertheless, after a leaf-reduction operation where $\theta_Z =x_4 \partial_{x_2} +  x_5 \partial_{x_3}+ \partial_{x_5}$, the foliated ideal sheaf $(M,\theta_Z,\mathcal{I},E)$ is still geometrically quasi-transverse, but does not satisfy condition \eqref{eq:aqt}
\end{itemize}
%condition $[II]$ follows from remark \ref{rk:InvReg} and if the decomposition in $[III]$ holds then $\theta_Z$ is necessarily the identically zero distribution, so the foliated ideal sheaf is analytically quasi-transverse in a trivial way.
\end{remark}

\noindent
Although the definition can not be verified without considering blowings-up and leaf-reduction operations (c.f. Remark \ref{rk:AQT}(3)) there is at least one interesting class of examples which satisfies these conditions:
\begin{lemma}
Let $(M,\theta,\mathcal{I},E)$ be a geometrically quasi-transverse foliated ideal sheaf and suppose that $X = V(\mathcal{I})$ is a regular analytic curve transverse to $E$. Then $(M,\theta,\mathcal{I},E)$ is analytically quasi-transverse.
\label{lem:daqt}
\end{lemma}
\begin{proof}
Denote by $\mathscr{F}$ denote the subset of geometrically quasi-transverse foliated ideal sheaves such that $X=V(\mathcal{I})$ is a regular curve transverse to $E$. We note that $\mathscr{F}$ is a sub-category of $GQT$ since it is closed by $\theta$-invariant blowings-up of order one and leaf-reduction operation. By the maximality of $AQT$ it is enough, therefore, to prove that \eqref{eq:aqt} is satisfied. 

%satisfies condition $[i-iii]$. We prove the result by induction on the generic leaf dimension $d$ of $(M,\theta,\mathcal{I},E) \in \mathscr{F}$. The result is clear when $d=0$, so let us fix $d>0$ and suppose that all elements in $\mathscr{F}$ with leaf dimension strictly smaller than $d$ are analytically quasi-transverse.

Fix a foliated ideal sheaf $(M,\mathcal{I},\theta,E)$ in $\mathscr{F}$ and a point $p$ in $X=V(\mathcal{I})$. Without loss of generality, there exists a coordinate system $(\pmb{x},y) = (x_1, \ldots, x_{m-1},y)$ centered at $p$ such that $\mathcal{I} = (x_1, \ldots, x_{m-1})$ (furthermore, if $p\in E$, we can assume that $E=(y=0)$). By the geometric quasi-transversality hypothesis and the fact that $\theta$ is tangent to $E$, we conclude that $\mathcal{J}:=(\theta[\mathcal{I}]+\mathcal{I})\cdot \mathcal{O}_p$ is either equal to $\mathcal{O}_p$ (and we are done), or $E=(y=0)$ and $\mathcal{J}=(x_1, \ldots, x_{m-1}, y^r)$ for some $r \in \mathbb{N}$. Now, from the fact that $\theta$ is tangent to $E$, we conclude that $\theta[y^r] \subset (y^r)$, which implies that $\theta[\mathcal{J}]\subset \mathcal{J}$. Thus $\theta^2[\mathcal{I}] \subset \theta[\mathcal{I}] + \mathcal{I}$, which proves that all elements in $\mathscr{F}$ satisfy \eqref{eq:aqt}.% We conclude that $\mathscr{F} \subset AQT$ by the maximality of $AQT$.

%Next, in the notation of Property $[II]$, the foliated ideal sheaf $(M,\theta_Z,\mathcal{I},E)$ is geometrically quasi-transverse because $\theta_Z \subset \theta$ and, therefore, belongs to $ \mathscr{F}$. This shows that $\mathscr{F}$ is closed by condition $[II]$, which finishes the proof.

%Finally, condition $[ii]$ holds because of remarks \ref{rk:InvReg} and \ref{rk:GQT}.
\end{proof}
\subsection{Main result}\label{ssec:mainTh}
\begin{theorem}
Let $(M,\theta,\mathcal{I},E)$ be an analytically quasi-transverse foliated ideal sheaf, then $(M,\theta,\mathcal{I},E)$ satisfies the generalized Flow-Box property.
\label{th:Rd}
\end{theorem}
\begin{proof}[Proof of Theorem \ref{th:Rdsimpl}]
%\begin{remark}
%\label{rk:Rd}
If $(M,\theta,\mathcal{I},E)$ is a geometrically quasi-transverse foliated ideal sheaf such that $X = V(\mathcal{I})$ is a regular analytic curve transverse to $E$, by Lemma \ref{lem:daqt} is is analytically quasi-transverse and by Theorem \ref{th:Rd} it satisfies the generalized Flow-Box property.
%\end{remark}
\end{proof}
%
%Before turning to the proof of the Theorem, we prove the following auxiliary result:
%
%
%
%
\begin{proof}[Proof of Theorem $\ref{th:Rd}$:] We prove by induction on the dimension of a \emph{generic leaf} $d$; the case that $d=0$ is trivial. So, assume by induction that the Theorem holds whenever $\theta$ has generic leaf dimension smaller or equal to $d-1$ and fix a singular distribution $\theta$ with generic leaf dimension $d$. Fix a point $p$ in $X=V(\mathcal{I})$ and a relatively compact open neighbourhood $M_0$ of $p$. By Proposition \ref{prop:HSimplificado} there exists a sequence of $\theta$-invariant blowings-up of order one:
\[
\begin{tikzpicture}
  \matrix (m) [matrix of math nodes,row sep=3em,column sep=4em,minimum width=2em]
  {(\widetilde{M},\widetilde{\theta},\widetilde{\mathcal{I}},\widetilde{E}) = (M_r,\theta_r,\mathcal{I}_r,E_r) & \cdots & (M_0,\theta_0,\mathcal{I}_0,E_0)\\};
  \path[-stealth]
    (m-1-1) edge node [above] {$\sigma_r$} (m-1-2)
    (m-1-2) edge node [above] {$\sigma_1$} (m-1-3);
\end{tikzpicture}
\]
such that (since $AQT$ is closed by $\theta$-admissible blowings-up) $\widetilde{\theta}[\widetilde{\mathcal{I}}] + \widetilde{\mathcal{I}} =\mathcal{O}_{\widetilde{M}}$. Since $AQT$ is a sub-category, $(\widetilde{M},\widetilde{\theta},\widetilde{\mathcal{I}},\widetilde{E})$ is analytically quasi-transverse. By Proposition \ref{prop:LT}, for each point $q$ in the pre-image of $p$, there exists a decomposition $\widetilde{\theta} = \widetilde{\theta}_Y + \widetilde{\theta}_Z$ valid in a sufficiently small open neighbourhood $\widetilde{U}$ of $q$ and (since $AQT$ is closed by leaf-reduction operation) $(\widetilde{U}, \widetilde{\theta}_Z,\widetilde{\mathcal{I}}|_{\widetilde{U}},\widetilde{E}|_{\widetilde{U}})$ is analytically quasi-transverse. By the induction hypothesis, $(\widetilde{U}, \widetilde{\theta}_Z,\widetilde{\mathcal{I}}|_{\widetilde{U}},\widetilde{E}|_{\widetilde{U}})$ satisfies the generalized Flow-Box property. By Proposition \ref{prop:LT2}, $(\widetilde{U}, \widetilde{\theta},\widetilde{\mathcal{I}}|_{\widetilde{U}},\widetilde{E}|_{\widetilde{U}})$ also satisfies the generalized Flow-Box property. Since $q$ was arbitrary, $(\widetilde{M}, \widetilde{\theta},\widetilde{\mathcal{I}},\widetilde{E})$ satisfies the generalized Flow-Box property and, finally, by Lemma \ref{lem:20}, the foliated ideal sheaf $(M_0,\theta_0,\mathcal{I}_0,E_0)$ satisfies the generalized Flow-Box property. Since the choice of $p \in X$ was arbitrary, we have proved the result.
\end{proof}
\addcontentsline{toc}{chapter}{Bibliography}
\bibliographystyle{alpha}

\end{document}